\documentclass[12pt]{article}

\usepackage{amssymb}
\usepackage{amsmath,amsthm}
\usepackage[dvips]{graphicx}
\usepackage{wrapfig}
\usepackage{bm}

\newtheorem{theorem}{Theorem}[section]
\newtheorem{corollary}[theorem]{Corollary}
\newtheorem{lemma}[theorem]{Lemma}
\newtheorem{proposition}[theorem]{Proposition}

\theoremstyle{definition}

\newtheorem{example}[theorem]{Example}
\theoremstyle{remark}

\title{The rank of a warping matrix}
\author{
Taira Akiyama \thanks{Department of Information and Computer Engineering, Gunma National College of Technology, 580 Toriba-cho, Maebashi-shi, Gunma, 371-8530, Japan.  }
\and 
Ayaka Shimizu \thanks{Department of Mathematics, Gunma National College of Technology, 580 Toriba-cho, Maebashi-shi, Gunma, 371-8530, Japan. Email: shimizu@nat.gunma-ct.ac.jp }
\and 
Ryohei Watanabe \thanks{Department of Information and Computer Engineering, Gunma National College of Technology, 580 Toriba-cho, Maebashi-shi, Gunma, 371-8530, Japan.  }
}
\date{\today}

\begin{document}

\maketitle

\begin{abstract}
The warping matrix has been defined for knot projections and knot diagrams by using warping degrees. 
In particular, the warping matrix of a knot diagram represents the knot diagram uniquely. 
In this paper we show that the rank of the warping matrix is one greater than the crossing number. 
We also discuss the linearly independence of knot diagrams by considering the warping incidence matrix. 
\end{abstract}


\section{Introduction}

The warping matrix is defined for oriented knot projections or diagrams on $S^2$ with all information of warping degree. 
It is shown in \cite{shimizu-wm} that there is one-to-one correspondence between oriented knot diagrams on $S^2$ and warping matrices. 
Hence the warping matrix can be used as a notation for oriented knots. 
However, it is not easy to calculate warping matrices for knot projections and diagrams with a large crossing number; 
the size of the warping matrix of a knot projection with $c$ crossings is $2^c \times 2c$, and that of a knot diagram is $(2^c -1) \times 2c$. 
One of the motivations on the study of warping matrix is to reduce the size of the matrix. 
In this paper we show the following theorem: 

\phantom{x}
\begin{theorem}
Let $P$ be an oriented knot projection on $S^2$, and $M(P)$ the warping matrix of $P$. 
We have the following equality: 
\begin{align*}
\mathrm{rank} M(P)=c(P)+1,
\end{align*}
where $c(P)$ is the crossing number of $P$. 
\label{mainthm}
\end{theorem}
\phantom{x}

\noindent Let $D$ be an oriented knot diagram on $S^2$, and $c(D)$ the crossing number of $D$. 
Let $\overline{M}(D)$ be the \textit{warping matrix of $D$ without signs}, which is mentioned concretely in Section 2. 
We also have the following theorem: 

\phantom{x}
\begin{theorem}
We have 
\begin{align*}
\mathrm{rank} \overline{M}(D)=c(D)+1. 
\end{align*}
\label{mainthm2}
\end{theorem}
\phantom{x}

\noindent The rest of this paper is organized as follows: 
In Section 2, we review the warping matrix. 
In Section 3, we prove Theorem \ref{mainthm} and \ref{mainthm2}. 
In Section 4, we define the warping incidence matrix of a knot diagram. 
In Section 5, we investigate linearly independence for knot diagrams.

\section{Warping matrix}

In this section we review the warping degree and warping matrix. 
See \cite{shimizu-wd} and \cite{shimizu-wm} for details. 
Let $D$ be an oriented knot diagram on $S^2$. 
Take a base point of $D$ avoiding crossing points. 
We denote by $D_b$ the pair of $D$ and $b$. 
A crossing point $p$ is a \textit{warping crossing point} of $D_b$ if we meet $p$ as an undercrossing first when we travel $D$ from $b$. 
The \textit{warping degree} $d(D_b)$, which is defined by Kawauchi in \cite{kawauchi}, of $D_b$ is the number of the warping crossing points of $D_b$. 

Let $P$ be an oriented knot projection on $S^2$. 
Take a base point at each edge of $P$, where edge means a part of $P$ between two crossings which has no crossings in the interior. 
Label them $b_1, b_2, \dots ,b_{2c}$ in order of traversal from an edge. 
From $P$, we obtain $2^c$ knot diagrams by giving over/under information at each crossing of $P$. 
We call them $D^1, D^2, \dots ,D^{2^c}$. 
The \textit{warping matrix} $M(P)$ of $P$ is the $2^c \times 2c$ matrix defined by: 
\begin{align*}
M(P)=( & a_{i j}), \\
\text{where } & a_{i j}=d(D^i _{b_j}). 
\end{align*}

\noindent An example is shown in Fig. \ref{image}. 
We consider warping matrices up to permutations on rows and cyclic permutations on columns. 
\begin{figure}[h]
\begin{center}
\includegraphics[width=130mm]{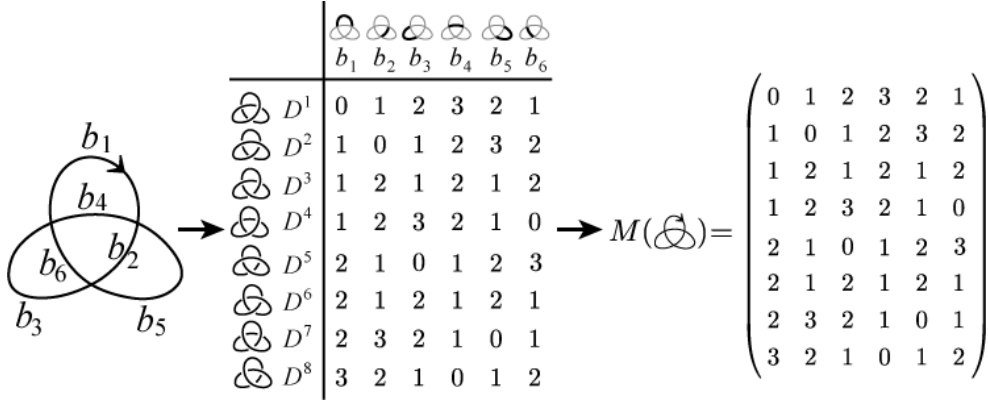}
\caption{Warping matrix.}
\label{image}
\end{center}
\end{figure}
Each row of $M(P)$ is said to be the \textit{warping degree sequence} of the corresponding diagram. 
As mentioned in Proposition 2.2 in \cite{shimizu-wm}, warping matrices of knot projections have the following properties: 
\begin{itemize}
\item On each row, the difference of two elements which are next to each other is one. 
\item On each column, $n$ appears 
$\begin{pmatrix}
c \\
n
\end{pmatrix}$
times $(n=0,1,2, \dots , c)$. 
\end{itemize}

\noindent The \textit{ou matrix} $U(P)$ of $P$ is the matrix obtained from $M(P)$ by the multiplication 
\begin{align*}
U(P)=M(P) \times A,
\end{align*}
where $A$ is the $2c \times 2c$ matrix as follows: 
\begin{align*}
A= \left(
\begin{array}{ccccc}
-1 & 0 & \ldots & 0 & 1 \\
1 & -1 & \ldots & 0 & 0 \\
0 & 1 & \ldots & 0 & 0 \\
\vdots & \vdots & \ddots & \vdots & \vdots \\
0 & 0 & \ldots & -1 & 0 \\
0 & 0 & \ldots & 1 & -1 \\
\end{array}
\right) .
\end{align*}

\noindent We give an example. 

\phantom{x}
\begin{example}
The warping matrix $M(P)$ of the knot projection $P$ depicted in Fig. \ref{twist} is
\begin{figure}[h]
\begin{center}
\includegraphics[width=30mm]{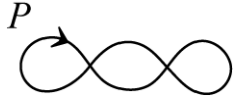}
\caption{Knot projection $P$.}
\label{twist}
\end{center}
\end{figure}

\begin{align*}
M(P)= \left(
\begin{array}{cccc}
0 & 1 & 2 & 1\\
1 & 2 & 1 & 2\\
1 & 0 & 1 & 0\\
2 & 1 & 0 & 1\\
\end{array}
\right) 
\end{align*}
and the ou matrix of $P$ is: 
\begin{align*}
U(P)= \left(
\begin{array}{cccc}
0 & 1 & 2 & 1\\
1 & 2 & 1 & 2\\
1 & 0 & 1 & 0\\
2 & 1 & 0 & 1\\
\end{array}
\right) 
\left(
\begin{array}{cccc}
-1 & 0 & 0 & 1  \\
1 & -1 & 0 & 0  \\
0 & 1 & -1 & 0 \\
0 & 0 & 1 & -1  \\
\end{array}
\right)
\end{align*}
\begin{align*}
 = \left(
\begin{array}{cccc}
1 & 1 & -1 & -1 \\
1 & -1 & 1 & -1 \\
-1 & 1 & -1 & 1 \\
-1 & -1 & 1 & 1 \\
\end{array}
\right) .
\end{align*}
\label{ou-matrix}
\end{example}
\phantom{x}

\noindent Each row of $U(P)$ corresponds to a knot diagram, each column corresponds to a crossing, and each element represents the over/under information where $1$ means over and $-1$ means under. 
Since we pass each crossing twice, there are just $c$ pairs of columns uniquely such that the sum of them is $\bm{0}$. 
For example, $U(P)$ of Example \ref{ou-matrix} has the pairs 1st and 4th, and 2nd and 3rd. 
From the pairing, we can recover the Gauss diagram of $P$. \\

We define the warping matrix without signs for knot diagrams. 
Let $D$ be an oriented knot diagram on $S^2$, and $P$ the knot projection obtained from $D$ by ignoring the over/under information. 
We define the warping matrix $\overline{M}(D)$ to be the matrix acquired from $M(P)$ by deleting the row of $D$. 
We can also obtain $\overline{M}(D)$ from the ordinary warping matrix $M(D)$ just by removing the signs of elements. 
For example, we have 

\newsavebox{\boxa}
\sbox{\boxa}{\includegraphics[width=11mm]{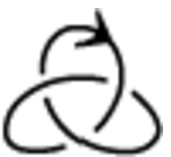}}
\begin{equation*}
\overline{M} ( \parbox[c]{11mm}{\usebox{\boxa}})=
\begin{pmatrix}
0 & 1 & 2 & 3 & 2 & 1 \\
1 & 0 & 1 & 2 & 3 & 2 \\
1 & 2 & 3 & 2 & 1 & 0 \\
2 & 1 & 0 & 1 & 2 & 3 \\
2 & 1 & 2 & 1 & 2 & 1 \\
2 & 3 & 2 & 1 & 0 & 1 \\
3 & 2 & 1 & 0 & 1 & 2 \\
\end{pmatrix}
\end{equation*}

\noindent (cf. Fig. \ref{image}).

\section{Proof of Theorem \ref{mainthm} and \ref{mainthm2}}

In this section we prove Theorem \ref{mainthm} and \ref{mainthm2}. 
We show the following lemma:  

\phantom{x}
\begin{lemma}
Let $c$ be an integer which is greater than 1. 
Let $a_1, a_2, \dots ,a_c, a_{c+1}$ be integers. 
We have the following equation. 
\begin{align}
\left|
\begin{array}{cccccc}
a_1 & -1 & 1 & \ldots & 1 & 1 \\
a_2 & 1 & -1 & \ldots & 1 & 1 \\
\vdots & \vdots & \vdots & \ddots & \vdots \\
a_{c-1} & 1 & 1 & \ldots & -1 & 1 \\
a_c & 1 & 1 & \ldots & 1 & -1 \\
a_{c+1} & -1 & -1 & \ldots & -1 & -1 
\end{array}
\right|
= -2^{c-1} \left( \sum _{i=1}^c a_i +(c-2) a_{c+1} \right)
\label{formula1}
\end{align}

\begin{proof}
By adding the $(c+1)$th row to the other rows, the equation (\ref{formula1}) is equivalent to: 

\begin{align}
\left|
\begin{array}{cccccc}
a_1 +a_{c+1} & -2 & 0 & \ldots & 0 & 0 \\
a_2 +a_{c+1} & 0 & -2 & \ldots & 0 & 0 \\
\vdots & \vdots & \vdots & \ddots & \vdots \\
a_{c-1} +a_{c+1} & 0 & 0 & \ldots & -2 & 0 \\
a_c +a_{c+1} & 0 & 0 & \ldots & 0 & -2 \\
a_{c+1} & -1 & -1 & \ldots & -1 & -1 
\end{array}
\right|
= -2^{c-1} \left( \sum _{i=1}^c a_i +(c-2) a_{c+1} \right)
\label{formula2}
\end{align}
We will show the equation (\ref{formula2}) by an induction on $c$. \\

\noindent $\bullet$ For $c=2$, we have 
\begin{align*}
\left|
\begin{array}{ccc}
a_1 +a_3 & -2 & 0 \\
a_2 +a_3 & 0 & -2 \\
a_3 & -1 & -1 
\end{array}
\right|
= -2(a_1 + a_2)
\end{align*}
Hence (\ref{formula2}) holds for $c=2$. \\

\noindent $\bullet$ Assume (\ref{formula2}) holds for $c=k-1$ for an integer $k$ greater than two. 
Now we consider the case $c=k$. 

\begin{align*}
\left|
\begin{array}{cccccc}
a_1 +a_{k+1} & -2 & 0 & \ldots & 0 & 0 \\
a_2 +a_{k+1} & 0 & -2 & \ldots & 0 & 0 \\
\vdots & \vdots & \vdots & \ddots & \vdots \\
a_{k-1} +a_{k+1} & 0 & 0 & \ldots & -2 & 0 \\
a_k +a_{k+1} & 0 & 0 & \ldots & 0 & -2 \\
a_{k+1} & -1 & -1 & \ldots & -1 & -1 
\end{array}
\right|
\end{align*}
The Laplace expansion along the $(k+1)$th column yields: 
\begin{align*}
(-1)^{k+(k+1)}\times (-2)
\left|
\begin{array}{ccccc}
a_1 +a_{k+1} & -2 & 0 & \ldots & 0 \\
a_2 +a_{k+1} & 0 & -2 & \ldots & 0 \\
\vdots & \vdots & \vdots & \ddots  \\
a_{k-1} +a_{k+1} & 0 & 0 & \ldots & -2 \\
a_{k+1} & -1 & -1 & \ldots & -1 
\end{array}
\right|
\end{align*}
\begin{align*}
+(-1)^{(k+1)+(k+1)}\times (-1)
\left|
\begin{array}{ccccc}
a_1 +a_{k+1} & -2 & 0 & \ldots & 0 \\
a_2 +a_{k+1} & 0 & -2 & \ldots & 0 \\
\vdots & \vdots & \vdots & \ddots \\
a_{k-1} +a_{k+1} & 0 & 0 & \ldots & -2 \\
a_k +a_{k+1} & 0 & 0 & \ldots & 0 
\end{array}
\right|
\end{align*}
Laplace expansion again along the $k$th row at the second term yields: 
\begin{align}
2
\left|
\begin{array}{ccccc}
a_1 +a_{k+1} & -2 & 0 & \ldots & 0 \\
a_2 +a_{k+1} & 0 & -2 & \ldots & 0 \\
\vdots & \vdots & \vdots & \ddots  \\
a_{k-1} +a_{k+1} & 0 & 0 & \ldots & -2 \\
a_{k+1} & -1 & -1 & \ldots & -1 
\end{array}
\right|
-(-1)^{k+1}(a_k+a_{k+1})
\left|
\begin{array}{cccc}
-2 & 0 & \ldots & 0 \\
0 & -2 & \ldots & 0 \\
\vdots & \vdots & \ddots & \vdots \\
0 & 0 & \ldots & -2
\end{array}
\right|
\label{formula3}
\end{align}
Remark the $(k, 1)$ element of the first matrix in (\ref{formula3}) is $a_{k+1}$. 
By assumption, (\ref{formula3}) is equal to
\begin{align*}
& 2 \left\{ -2^{(k-1)-1} \left( \sum _{i=1}^{k-1} a_i + (k-1-2) a_{k+1} \right) \right\} - (-1)^{k+1}(a_k+a_{k+1})(-2)^{k-1}\\
& =-2^{k-1} \left( \sum _{i=1}^{k-1} a_i + (k-3) a_{k+1} \right) -2^{k-1}(a_k+a_{k+1})\\
& =-2^{k-1} \left( \sum _{i=1}^{k-1} a_i + (k-3) a_{k+1} +a_k +a_{k+1} \right) \\
& =-2^{k-1} \left( \sum _{i=1}^k a_i +(k-2)a_{k+1} \right) .
\end{align*}
Hence (\ref{formula2}) holds. 
\end{proof}
\label{ranklem}
\end{lemma}
\phantom{x}

\noindent By multiplying by $-1$ the $(c+1)$th column, we obtain the following corollary from Lemma \ref{ranklem}:

\phantom{x}
\begin{corollary}
We have 
\begin{align}
\left|
\begin{array}{cccccc}
a_1 & -1 & 1 & \ldots & 1 & -1 \\
a_2 & 1 & -1 & \ldots & 1 & -1 \\
\vdots & \vdots & \vdots & \ddots & \vdots \\
a_{c-1} & 1 & 1 & \ldots & -1 & -1 \\
a_c & 1 & 1 & \ldots & 1 & 1 \\
a_{c+1} & -1 & -1 & \ldots & -1 & 1 
\end{array}
\right|
= 2^{c-1} \left( \sum _{i=1}^c a_i +(c-2) a_{c+1} \right) .
\label{formula4}
\end{align}
\label{lankcor}
\end{corollary}

Now we prove Theorem \ref{mainthm}. 

\phantom{x}
\noindent \textit{Proof of Theorem \ref{mainthm}.} \ 
For a $2^c \times 2c$ matrix $M(P)=( \bm{a}_1, \bm{a}_2, \dots ,\bm{a}_{2c})$, subtract $\bm{a}_{2c-1}$ from $\bm{a}_{2c}$, $\bm{a}_{2c-2}$ from $\bm{a}_{2c-1}$, $\bm{a}_{2c-3}$ from $\bm{a}_{2c-2}$, $\dots$, and $\bm{a}_1$ from $\bm{a}_2$. 
Let $( \bm{a}, \bm{v_1}, \bm{v_2}, \dots ,\bm{v_{2c-1}})$ be the matrix we obtain by the procedure. 
Note that $( \bm{v_1}, \bm{v_2}, \dots ,\bm{v_{2c-1}})$ is a submatrix of $U(P)$. 
By the property of ou matrix, there are just $(c-1)$ pairs of columns $\bm{v}_i$ and $\bm{v}_j$ uniquely such that $\bm{v}_i + \bm{v}_j = \bm{0}$. 
For each pair, add one to the other. 
By reordering some columns $\bm{v}_i$, we have 
\begin{align*}
(\bm{a}, \bm{v_{\varphi _1}}, \bm{v_{\varphi _2}}, \dots ,\bm{v_{\varphi _c}}, \bm{0}, \bm{0}, \dots ,\bm{0} ).
\end{align*}
Hence
\begin{align*}
\mathrm{rank} M(P) \le c(P)+1. 
\end{align*}
Next we show that $\bm{a}, \bm{v_{\varphi _1}}, \bm{v_{\varphi _2}}, \dots ,\bm{v_{\varphi _c}}$ are linearly independent. 
Since the columns $\bm{v_{\varphi _1}}, \bm{v_{\varphi _2}}, \dots ,\bm{v_{\varphi _c}}$ correspond to all the $c$ crossings, and the rows correspond to all the over/under information, we can obtain the following $(c+1) \times (c+1)$ submatrix of $( \bm{a}, \bm{v_{\varphi _1}}, \bm{v_{\varphi _2}}, \dots ,\bm{v_{\varphi _c}} )$ by reordering some rows
\begin{align*}
\left(
\begin{array}{cccccc}
a_1 & -1 & 1 & \ldots & 1 & 1 \\
a_2 & 1 & -1 & \ldots & 1 & 1 \\
\vdots & \vdots & \vdots & \ddots & \vdots \\
a_{c-1} & 1 & 1 & \ldots & -1 & 1 \\
a_c & 1 & 1 & \ldots & 1 & -1 \\
a_{c+1} & -1 & -1 & \ldots & -1 & -1 
\end{array}
\right)
\end{align*}
where $a_i$ is an element of $\bm{a}$ $(i=1,2,\dots ,c+1)$. 
By Lemma \ref{ranklem}, the determinant of the submatrix is 
\begin{align*}
-2^{c-1} \left( \sum _{i=1}^c a_i +(c-2) a_{c+1} \right) . 
\end{align*}
Remark that elements of $\bm{a}$ are all non-negative and there are at most one $0$. 
Hence the determinant is non-zero. 
Therefore $\bm{a}, \bm{v_{\varphi _1}}, \bm{v_{\varphi _2}}, \dots ,\bm{v_{\varphi _c}}$ are linearly independent. 
Hence 
\begin{align*}
\mathrm{rank} M(P) = c(P)+1. 
\end{align*}
\hfill$\square$
\phantom{x}

\noindent We prove Theorem \ref{mainthm2}. 

\phantom{x}
\noindent \textit{Proof of Theorem \ref{mainthm2}.} \ 
Similar to Theorem \ref{mainthm} except that $\overline{M}(D)$ has one row fewer than $M(P)$. 
Remark that we can obtain at least one of the following two submatrices 
\begin{align*}
\left(
\begin{array}{cccccc}
a_1 & -1 & 1 & \ldots & 1 & 1 \\
a_2 & 1 & -1 & \ldots & 1 & 1 \\
\vdots & \vdots & \vdots & \ddots & \vdots \\
a_{c-1} & 1 & 1 & \ldots & -1 & 1 \\
a_c & 1 & 1 & \ldots & 1 & -1 \\
a_{c+1} & -1 & -1 & \ldots & -1 & -1 
\end{array}
\right)
\ \text{and} \ 
\left(
\begin{array}{cccccc}
a_1 & -1 & 1 & \ldots & 1 & -1 \\
a_2 & 1 & -1 & \ldots & 1 & -1 \\
\vdots & \vdots & \vdots & \ddots & \vdots \\
a_{c-1} & 1 & 1 & \ldots & -1 & -1 \\
a_c & 1 & 1 & \ldots & 1 & 1 \\
a_{c+1} & -1 & -1 & \ldots & -1 & 1 
\end{array}
\right) , 
\end{align*} 
which are the matrices of Lemma \ref{ranklem} and Corollary \ref{lankcor}. 
\hfill$\square$
\phantom{x}

\section{Warping incidence matrix}

In this section we define the warping incidence matrix for oriented knot diagrams. 
Let $D$ be an oriented knot diagram on $S^2$ with $c$ crossings. 
Label the edges $e_1, e_2, \dots , e_{2c}$ in order of traversal from an edge. 
Label the crossings $v_1, v_2, \dots ,v_c$. 
The \textit{warping incidence matrix} $m(D)$ of $D$ is defined as follows: 

\begin{eqnarray*}
m(D)=( a_{i j}), 
\end{eqnarray*}

\begin{eqnarray*}
\text{where }   a_{i j}=
\left \{ \begin{array}{ll}
1 & \text{(} v_i \text{ is a warping crossing point of } D_{b_j} \text{)}\\
0 & \text{(} v_i \text{ is not a warping crossing point of }D_{b_j} \text{)}, 
\end{array} \right.
\end{eqnarray*}

\noindent where $b_j$ is a base point on $e_j$. 
An example is shown in Fig. \ref{image2}. 
We consider warping incidence matrices up to permutations on rows and cyclic permutations on columns. 
\begin{figure}[h]
\begin{center}
\includegraphics[width=130mm]{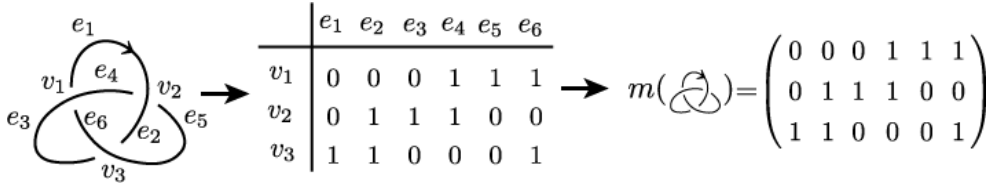}
\caption{Warping incidence matrix.}
\label{image2}
\end{center}
\end{figure}

\noindent By definition, we have the following: 

\phantom{x}
\begin{proposition}
The sum of all the rows of the warping incidence matrix of a knot diagram $D$ is the warping degree sequence of $D$. 
\end{proposition}
\phantom{x}

\noindent We show the following proposition. 

\phantom{x}
\begin{proposition}
Let $m(D)=(a_{i  j})$ be the warping incidence matrix of a knot diagram $D$ with $c(D)=c$. 
For any integer $k \in \{ 1,2, \dots ,c \}$, there exists an integer $l \in \{ 1,2, \dots , 2c \}$ uniquely such that 
\begin{eqnarray}
\left \{ \begin{array}{l}
a_{k \ l-1}=0, \\
a_{k \ l}=1 \text{ and} \\
a_{i \ l-1}=a_{i \ l} \ (i \neq k ).
\end{array} \right.
\label{kl-formula}
\end{eqnarray}
\label{kl}
\end{proposition}

\begin{proof}
See Fig. \ref{crossings}. 
At the $k$th row in $m(D)$, $1$ appears from $l$th column through $m$th column because the corresponding edges have the crossing $v_k$ as a warping crossing point, 
whereas $0$ appears from $(m+1)$th column through $(l-1)$th column. 
Hence we have $a_{k \ l-1}=0$ and $a_{k \ l}=1$. 
Since the two edges $e_{l-1}$ and $e_l$ have the same warping crossing points except at $v_k$, we have $a_{i \ l-1}=a_{i \ l}$ for $i \ne k$. 
\begin{figure}[h]
\begin{center}
\includegraphics[width=60mm]{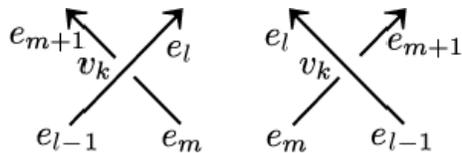}
\caption{A crossing and four edges.}
\label{crossings}
\end{center}
\end{figure}
\end{proof}
\phantom{x}

\noindent We have the following corollaries: 

\phantom{x}
\begin{corollary}
Let $D$ be a knot diagram, and let $D^i$ be the diagram obtained from $D$ by a crossing change at the crossing $v_i$. 
Then, $m(D^i)$ is obtained from $m(D)$ by switching 0 and 1 at the $i$th row. 
\label{cc-cor}
\end{corollary}
\phantom{x}

\begin{corollary}
We can obtain the Gauss diagram of $D$ without signs from $m(D)$. 
\end{corollary}
\phantom{x}

\noindent We show the following lemma: 

\phantom{x}
\begin{lemma}
Let $D$ be a knot diagram with $c(D)=c$. 
Let $m(D)$ be the warping incidence matrix of $D$, and $\bm{v_i}$ the $i$th row of $m(D)$ $(i=1,2, \dots ,c)$. 
Then $\bm{v_1}, \bm{v_2}, \dots ,\bm{v_c}$ and $\bm{1}$ are linearly independent, where $\bm{1}$ is the row vector $(1 1 \dots 1)$ with length $2c$. 
\label{keylem}
\end{lemma}

\begin{proof}
Let 
\begin{align}
\lambda _1 \bm{v_1} + \lambda _2 \bm{v_2} + \dots +\lambda _c \bm{v_c} + \lambda _{c+1} \bm{1}  = \bm{0}. 
\label{l-combination}
\end{align}
For each $k \in \{1,2, \dots ,c \}$, there exists an integer $l$ satisfying (\ref{kl-formula}) in Proposition \ref{kl}. 
Since the $l$th and $(l-1)$th elements of (\ref{l-combination}) are zero, the difference of them, which is $\lambda _k$, is also zero. 
Hence we have $\lambda _1 = \lambda _2 = \dots = \lambda _c =0$ and therefore $\lambda _{c+1} =0$. 
\end{proof}
\phantom{x}

\noindent From Lemma \ref{keylem}, we have the following corollary: 

\phantom{x}
\begin{corollary}
We have 
\begin{align*}
\mathrm{rank} m(D)=c(D). 
\end{align*}
\end{corollary}

\section{Linearly independent diagrams}
 
By Theorem \ref{mainthm}, each warping matrix $M(P)$ of a knot projection $P$ has $(c(P)+1)$ linearly independent rows. 
We say that knot diagrams obtained from a same knot projection $P$ are \textit{linearly independent} if the corresponding rows in $M(P)$ are linearly independent. 
We have the following theorem. 

\phantom{x}
\begin{theorem}
Let $D$ be a knot diagram with $c$ crossings. 
Then $D$ and all the $c$ diagrams obtained from $D$ by a single crossing change are linearly independent. 
\label{prop-independent}
\end{theorem}
\phantom{x}

\begin{proof}

Let $m(D)$ be the warping incidence matrix of $D$, and $\bm{v_i}$ be the $i$th row of $m(D)$. 
As mentioned in Section 4, the warping degree sequence $s(D)$ of $D$ is obtained by $\bm{v_1}+\bm{v_2}+ \dots +\bm{v_c}$. 
Let $D^i$ be the knot diagram obtained from $D$ by a crossing change at $v_i$. 
By Corollary \ref{cc-cor}, we have $s(D^i)=\bm{v_1}+\bm{v_2}+ \dots + \bm{v_{i-1}}+( \bm{1} - \bm{v_i})+\bm{v_{i+1}}+ \dots +\bm{v_c}$. 
We will prove that $s(D), s(D^1), s(D^2), \dots $ and $s(D^c)$ are linear independent. 
By subtracting $s(D)$ from the others, it is sufficient to show that $(\bm{v_1}+\bm{v_2}+ \dots +\bm{v_c})$, $( \bm{1}-2 \bm{v_1})$, $( \bm{1}-2 \bm{v_2}), \dots $ and $( \bm{1}-2 \bm{v_c})$ are linearly independent. 
Let 
\begin{align*}
\lambda _o (\bm{v_1}+\bm{v_2}+ \dots +\bm{v_c}) +\lambda _1 ( \bm{1}-2 \bm{v_1}) + \dots + \lambda _c( \bm{1}-2 \bm{v_c})=\bm{0}. 
\end{align*}
Then we have 
\begin{align*}
( \lambda _0 -2\lambda _1)\bm{v_1} + ( \lambda _0 -2\lambda _2)\bm{v_2} + \dots +( \lambda _0 -2\lambda _c)\bm{v_c} + ( \lambda _1 +\lambda _2 + \dots +\lambda _c)\bm{1}=0.
\end{align*}
By Lemma \ref{keylem}, we have $\lambda _0 -2\lambda _1=\lambda _0 -2\lambda _2= \dots =\lambda _0 -2\lambda _c=\lambda _1 +\lambda _2 + \dots +\lambda _c=0$. 
Hence $\lambda _0=\lambda _1= \dots = \lambda _c=0$. 
\end{proof}

\section*{Acknowledgments}
\noindent A.~S.~expresses gratitude to Professors Kazuaki Kobayashi and Takako Kodate for valuable discussions and suggestions.
T.~A., A.~S.~and R.~W.~are grateful to Yukari Fuseda, Yuui Onoda and Professor Yoshiro Yaguchi for valuable discussions.

\end{document}